\documentclass[12pt,a4paper]{amsart}
\usepackage[top=25truemm,bottom=25truemm,left=20truemm,right=20truemm]{geometry}

\usepackage{amsmath}  %
\usepackage{amsfonts} %
\usepackage{amsmath} %
\usepackage{mathrsfs} %
\usepackage{amssymb} %
\usepackage{amsthm} %
\usepackage{stmaryrd} %
\usepackage{enumerate} %
\usepackage{enumitem} %
\usepackage{turnstile} %
\usepackage{tikz-cd} %
\usepackage{hyperref} %
\usepackage{color} %

\setlist[itemize]{leftmargin=20pt,label=$-$}
\setlist[enumerate]{leftmargin=20pt,label=(\roman*)}

\hypersetup{
    colorlinks=true,       
    linkcolor=red,          
    citecolor=red,        
    filecolor=red,      
    urlcolor=red           
}

\newtheoremstyle{customthm}
{}
{}
{\itshape}
{}
{\scshape}
{.}
{.5em}
{}

\newtheoremstyle{customdef}
{}
{}
{}
{}
{\scshape}
{.}
{.5em}
{}

{
\theoremstyle{customthm} %
\newtheorem{theorem}{Theorem}[section] %
}
{\theoremstyle{customdef} %
  \newtheorem{definition}[theorem]{Definition} %
  \newtheorem{proposition}[theorem]{Proposition} %
  \newtheorem{corollary}[theorem]{Corollary} %
  \newtheorem{lemma}[theorem]{Lemma} %
  \newtheorem{example}[theorem]{Example} %
}
{
\theoremstyle{remark}
\newtheorem{remark}[theorem]{Remark}
}

\begin{document}

\title
  [Varieties and equations]         
  {A categorical view of varieties and equations} 

\author{Jose Avila}
\email{jose.avila@correounivalle.edu.co}  

\thanks{
The author is grateful to Guillermo Ortiz and Sergio Troncoso for their useful comments and suggestions.}

\begin{abstract}
    We present a common framework to study varieties  in great generality from a categorical point of view. The main application of this study is in the setting of algebraic categories, where we introduce Birkhoff varieties which are essentially subvarieties of algebraic categories, and we get a generalization of Birkhoff's variety theorem. In particular, we show that Birkhoff varieties are coreflexive equalizers. The key of this generalization is to give a more general concept of equation for subvarieties of algebraic categories. In order to get our characterization of Birkhoff varieties, we study inserters over algebraic categories, where we generalize some well-known results of algebras for finitary endofunctors over \textit{Set}. By duality, we obtain a characterization of cosubvarieties of coalgebraic categories. Surprisingly, these cosubvarieties turn to be varieties according to our theory of varieties. 
\end{abstract}

\maketitle

\section{Introduction}

Varieties could be defined as the set of solutions of a system of equations. Equations appear at everywhere in mathematics. In general, an equation is a formula $\varphi(x) \approx \psi(x)$ where $\varphi(x),\psi(x)$ are expressions in function on the same object $x$, which belongs to some domain of interest. In other words, an equation is a pair of parallel functions. Given a system of equations $$E=\{\varphi_i,\psi_i : S \to B_i \}_{i \in I},$$ the variety defined by $E$ is 
\[
V = \{ a \in S \mid \varphi_i(a) = \varphi_i(a) \, \text{for all} \, i \in I \} . 
\]
Observe that if $W$ is a collection of solutions of the equations $\varphi_i \approx \psi_i$, then $W \subset V$. More generally, let $\iota: V \hookrightarrow S$ be the inclusion of $V$ into $S$. If $f:C \to S$ is a function such that $\varphi_i(f(x)) = \psi_i(f(x))$ for all $i \in I$ and $x \in C$, i.e., $\varphi_i f = \psi_i f$ for all $i \in I$, then there exists a unique function $g: C \to V$ such that $f = \iota g$. This universal property of the inclusion $\iota$ with respect to $E$ defines $V$. Indeed, any other function $\zeta: V' \to S$ with this property is a re-indexation of $V$, what we mean is that $\zeta$ is one-to-one and its image is $V$. To the functions with this property we could to call them variety functions defined by $E$.

The considerations above apply to the category $Set$ of sets and functions. However, it is already clear, by the above discussion, how to generalize these concepts to any other category. We do this in Section \ref{Abstract equations and varieties}. This study naturally extends to covarieties by duality. It must been said, the aim of this note is to give a common framework to study varieties, in search to better understanding of these objects. 

A similar systematic study of equations is in unification theory, a subject of computer science and logic, see \cite{EDER198531,Lassez88}. Most general unifiers are closed related to variety morphisms: both objects represent the set of all solutions of a given system of equations, and any two most general unifiers (resp. variety morphisms) of a same system of equations are isomorphic. Although, in general, most general unifiers fail to give a unique representation of the solutions, for example, ${\theta=\{x \mapsto a,\, y \mapsto b\}}$ is a most general unifier of the equation ${f(x,b) \approx f(a,y)}$, and $\theta$ is idempotent, so $\theta = \theta \circ \theta$ and $\theta = \theta \circ \iota$, but $\theta \neq \iota$, where $\iota$ is the identity substitution.

There are many classes of well-known varieties. We give special attention to varieties of algebras. In universal algebra, varieties are equationally presentable collections of (one or many-sorted) algebras. In the case of one-sorted algebras, varieties are characterized by Birkhoff's variety theorem \cite{birkhoff_1935}, as classes of algebras closed under homomorphic images, subalgebras, and products, also called HSP classes of algebras.

A significant step in generalizing these algebraic structures is in the context of algebraic categories. Algebraic theories and its algebras were defined by F. W. Lawvere in his doctoral dissertation \cite{Lawvere63}. An excellent modern account of algebraic theories, and general algebra, is given in \cite{adámek_rosický_vitale_lawvere_2010}. The following characterizations, see \cite[Theorem 6.9]{adámek_rosický_vitale_lawvere_2010}, of algebraic categories correspond to generalizations of the concepts introduced by Lawvere. Let $\mathcal{A}$ be a locally small category. Then the following conditions are equivalent:

\begin{itemize}
    \item  $\mathcal{A}$ is algebraic, i.e.,  equivalent to $Alg \, \mathcal{T}$, for some algebraic theory $\mathcal{T}$. 

    \item $\mathcal{A}$ is cocomplete and has a set $\mathcal{G}$ of perfectly presentable objects such that every object of $\mathcal{A}$ is a sifted colimit of objects of $\mathcal{G}$. 

    \item $\mathcal{A}$ is cocomplete and has a strong generator $\mathcal{G}$ consisting of perfectly presentable objects. 
\end{itemize}

Recall that an algebraic theory is a small category with finite products. We denote by \textbf{AlgTh} to the category of algebraic theories. The morphisms between algebraic theories are finite product preserving functors. Likewise, we denote by \textbf{AlgCat} to the category of algebraic categories, with algebraic functors as morphisms, where a functor between algebraic categories is algebraic if preserves limits and sifted colimits. Algebraic categories have many notable properties. They are locally finitely presentable, in particular, they are complete, wellpowered and cowellpowered. They have regular factorizations, and regular epimorphisms in such categories are stable under pullbacks and products. The perfectly presentable (also called strongly presentable) objects of an algebraic category are precisely the finitely presentable regular projective objects. 

In algebraic categories there is also a Birkhoff's variety theorem, see \cite[Theorem 10.22]{adámek_rosický_vitale_lawvere_2010}. The subvarieties of an algebraic category $\mathcal{A}$ are precisely the full subcategories of $\mathcal{A}$ closed under regular quotients, subalgebras, products, and direct unions. The assumption of closure under directed unions cannot be
omitted, see \cite[Example 10.23]{adámek_rosický_vitale_lawvere_2010} for a counterexample in the case $\mathcal{A} = Set^\mathbb{N}$. However, subvarieties are defined in first place by equations, but equations are defined only in algebraic categories of the form $Alg \, \mathcal{T}$. Here, we present an alternative definition of equation 
which does not depend on the presentation of $\mathcal{A}$, and corresponds to the theory presented in Section \ref{Abstract equations and varieties}. 

This lead us to introduce Birkhoff varieties and Lawvere covarieties, which are defined by Birkhoff and Lawvere equations, respectively.  A Birkhoff equation, see Definition \ref{birkhoff equation}, is an equation $P \approx Q$ in \textbf{AlgCat} such that there exists an algebraic, faithful, conservative, and amnestic functor $U$ such that $UP = UQ$. On the other hand, a Lawvere equation, see Definition \ref{Lawvere equation}, is an equation $P \approx Q$ in \textbf{AlgTh} such that there exists a morphism of theories $U$ surjective on objects such that $PU = QU$. Our main results are the following characterizations. 

{
\theoremstyle{customthm}
\newtheorem{BirkhoffVar}[theorem]{Characterization of Birkhoff varieties}
}
\begin{BirkhoffVar}\label{Characterization of Birkhoff varieties}
    Let $F:\mathcal{B} \to \mathcal{A}$ be an algebraic functor. Then the following conditions are equivalents:
    \begin{itemize}
        \item $F$ is a Birkhoff variety, i.e., a general solution of some system of Birkhoff equations.
        \item $F$ is isomorphic to the inclusion $\mathcal{V} \hookrightarrow \mathcal{A}$, for some subvariety $\mathcal{V}$ of $\mathcal{A}$.
        \item $F$ is a coreflexive equalizer of some coreflexive Birkhoff equation. 
    \end{itemize}
\end{BirkhoffVar}

{
\theoremstyle{customthm}
\newtheorem{LawvereVar}[theorem]{Characterization of Lawvere covarieties}
}
\begin{LawvereVar}\label{Characterization of Lawvere covarieties}
    Let $M:\mathcal{T} \to \mathcal{Q}$ be a morphisms of algebraic theories. Then the following conditions are equivalents:
    \begin{itemize}
        \item $M$ is a Lawvere covariety, i.e., a general cosolution of some cosystem of Lawvere equations.
        \item $M$ is isomorphic to the canonical morphism $\mathcal{T} \to \mathcal{T}/\sim$, for some congruence $\sim$ (in the sense of \cite[Definition 10.4]{adámek_rosický_vitale_lawvere_2010}) on $\mathcal{T}$.
        \item $M$ is a reflexive coequalizer of some reflexive Lawvere equation.
    \end{itemize}
\end{LawvereVar}

Moreover, we prove that every system of Birkhoff equations (resp. every cosystem of Lawvere equations) has a general solution (resp. a general cosolution). Also, we give another characterization of Birkhoff varieties in terms of Lawvere covarieties, see Theorem \ref{Lawvere Birkhoff varieties}. These similarities between Birkhoff varieties and Lawvere covarities are to be expected since \textbf{AlgCat} and \textbf{AlgTh} are almost dual to each other, see \cite[Theorem 9.15]{adámek_rosický_vitale_lawvere_2010}.

In order to get our characterization of Birkhoff varieties, we study inserters over algebraic categories. Inserters have been studied in \cite{adamek_rosicky_1994} for the quasicategory of all categories and all functors, and more generally in \cite{BIRD19891} for $2$-categories. We consider them only in the \textbf{Cat} setting, for \textbf{Cat} the non-locally small category of all locally small categories and all functors. Inserters are a straight generalization of algebras and coalgebras for endofunctors. Both algebras and coalgebras are very well-known, and they are dual to each other. Algebras (for an endofunctor) model common algebraic structures like one-sorted algebras, and coalgebras model structures like deterministic automata. In particular, algebras for finitary endofunctors over $Set$ have been studied in \cite[Chapter 12]{adámek_rosický_vitale_lawvere_2010}. The results there have generalizations, as noted at the beginning and the end of that Chapter, for locally finitely presentable  and algebraic categories. We recover some of these results for inserters over algebraic categories in Section \ref{inserters section}. The key idea is to give conditions under which inserters are concretely isomorphic to algebras or coalgebras, see Theorem \ref{insertes algebras and coalgebras}.

Our characterization of Birkhoff varieties is easily generalizable for coalgebraic categories. We just state this result. A concise definition of the category \textbf{CoAlgCat} of coalgebraic categories is the following: a locally small category $\mathcal{A}$ is coalgebraic if only if $\mathcal{A}^{\operatorname{op}}$ is algebraic, a functor $F:\mathcal{A} \to \mathcal{B}$ between coalgebraic categories is coalgebraic if only if $F^{\operatorname{op}}:\mathcal{A}^{\operatorname{op}} \to \mathcal{B}^{\operatorname{op}}$ is algebraic. Observe that the opposite functor
\[
F:\mathcal{A} \to \mathcal{B} \quad \longmapsto \quad F^{\operatorname{op}}: \mathcal{A}^{\operatorname{op}} \to \mathcal{B}^{\operatorname{op}}
\]
is an isomorphism between \textbf{CoAlgCat} and \textbf{AlgCat}. Let $\mathcal{V}$ be a full subcategory a coalgebraic category $\mathcal{A}$, we call $\mathcal{V}$ a cosubvariety of $\mathcal{A}$ if $\mathcal{V}^{\operatorname{op}}$ is a subvariety of $\mathcal{A}^{\operatorname{op}}$. A co-Birkhoff equation is an equation $P \approx Q$ in \textbf{CoAlgCat}  such that there exists a coalgebraic, faithful, conservative, and amnestic functor $U$ such that $UP = UQ$. Thus, the dual of Characterization of Birkhoff varieties \ref{Characterization of Birkhoff varieties} is:

{
\theoremstyle{customthm}
\newtheorem{coBirkhoffVar}[theorem]{Characterization of co-Birkhoff varieties}
}
\begin{coBirkhoffVar}\label{co-birkhoff varieties}
        Let $F:\mathcal{B} \to \mathcal{A}$ be a coalgebraic functor. Then the following conditions are equivalents:
    \begin{itemize}
        \item $F$ is a co-Birkhoff variety, i.e., a general solution of some system of co-Birkhoff equations.
        \item $F^{\operatorname{op}}$ is a Birkhoff variety. 
        \item $F$ is isomorphic to the inclusion $\mathcal{V} \hookrightarrow \mathcal{A}$, for some cosubvariety $\mathcal{V}$ of $\mathcal{A}$.
        \item $F$ is a coreflexive equalizer of some coreflexive co-Birkhoff equation. 
    \end{itemize}
\end{coBirkhoffVar}

In summary, what we have actually done in Characterization of Birkhoff varieties \ref{Characterization of Birkhoff varieties} is to characterize, in a precise sense (see Inverse main problem of varieties \ref{inversemainproblem} and Proposition \ref{unique solution to inverse problem}), the class of all equations in the context of \textbf{AlgCat}, which define subvarieties of algebraic categories. This gives us a generalization of Birkhoff's variety theorem \cite[Theorem 10.22]{adámek_rosický_vitale_lawvere_2010}. From this, we get a generalization of the dual of Birkhoff's variety theorem given in \cite{awodey2000coalgebraic}.  

\begin{remark}
    Our definition of covariety differs to the usual definition found in the literature, see for example \cite{awodey2000coalgebraic,GUMM200171}. What they call covarieties  in those papers is what we call cosubvarieties. We define covarieties, see Definition \ref{covariety definition}, as the categorical dual of varieties, see Definition \ref{variety definition}. Surprisingly, covarieties in the sense of \cite{awodey2000coalgebraic,GUMM200171}, turn to be varieties according to our definition of varieties, for the setting of coalgebraic categories. This is due to Characterization of co-Birkhoff varieties \ref{co-birkhoff varieties}. 
\end{remark}

\begin{remark}
    In no way this note is a complete account of varieties in their utmost generality. We mention many other examples, see Section \ref{examples}, like varieties for polynomial equations or differential equations.  Of course, there are still many open questions to be solved, we list some of these in the last section.
\end{remark}

\section{Abstract equations and varieties}\label{Abstract equations and varieties}

In this section we consider a (non-necessarily locally small) category $\mathcal{C}$.

\begin{definition}
    An \textbf{equation} is a pair $f,g$ of parallel arrows, i.e., a pair of morphisms with same source and target, which we denote by $f \approx g$. The \textbf{domain} of $f \approx g$ is the common domain of $f$ and $g$. A \textbf{system of equations} is a non-empty set of equations with same domain.
\end{definition}

\begin{definition}
    Let $E$ be a system of equations defined on $A$. A \textbf{solution} of (the equations in) $E$ is an arrow $a$ with target $A$ such that $f a = g a$ for all equations $f \approx g$ in $E$. We write $a \models E$ if $a$ is a solution of $E$. More generally, if $S$ is a set of solutions of $E$, we write $S \models E$. 
\end{definition}

\begin{definition}\label{variety definition}
    A \textbf{general solution} of a system of equations $E$ is a solution $v$ of $E$ with the property that for each solution $a$ of $E$ there exists a unique morphism $\tilde{a}$ such that $a = v \Tilde{a}$. We write $v \sdtstile{\text{g.s}}{} E$ to denote $v$ is a general solution of $E$.
    A \textbf{variety (morphism)} is a general solution of some system of equations. 
\end{definition}

It is clear that every variety is a monomorphism, and every regular monomorphism is a variety. Note that general solutions of a system of equations $E$ correspond to limits for the diagram in $\mathcal{C}$ defined by $E$. A general solution of $E$ gives a unique representation to the set of all its solutions.   

Given morphisms $f,g$, we write $f \leq g$ if only if $f = g h$ for some $h$. We said $f$ and $g$ are isomorphic if there exists an isomorphism $h$ such that $f = g h$. In particular, if $f$ and $g$ are monomorphisms, then $f$ and $g$ are isomorphic if only if $f \leq g$ and $g \leq f$. Similarly, if $E$ and $K$ are systems of equations with same domain, we write $E \implies K$ if every solution of $E$ is a solution of $K$, and we said $E$ is equivalent to $K$ if $E$ and $K$ have the same solutions, i.e., $E \implies K$ and $K \implies E$. Given two classes $\Sigma, \Theta$ of equations, we said $\Sigma$ and $\Theta$ are equivalents if each equation in $\Sigma$ is equivalent to some equation in $\Theta$, and each equation in $\Theta$ is equivalent to some equation in $\Sigma$. 

\begin{remark}
    It is clear that there is only one general solution for a given system of equations, up to isomorphism.
\end{remark}

Arrows act on equations by right composition, and this action extents to an action over system of equations. Explicitly, let $E$ be a system of equations with domain $A$ and let $g:B \to A$ be an arrow with target $A$, then 
$
Eg = \{ pg \approx qg \mid p \approx q \in E \} . 
$

Let $\Sigma$ be a class of equations. A $\Sigma$-variety is a general solution to some system of $\Sigma$-equations. We write $a \sdtstile{}{\Sigma} E$ if $a \models E$ and $E$ is a system of $\Sigma$-equations.

\begin{proposition} The following statements holds:
    \begin{enumerate}
        \item Let $E$ and $K$ be systems of equations with same domain and suppose that $v \sdtstile{\text{g.s}}{} E$ and $w \sdtstile{\text{g.s}}{} K$. Then $E \implies K$ is equivalent to $v \leq w$.

        \item Let $E$ be a system of equations with domain $A$ and $g:B \to A$. Then $a \models Eg$ if only if $ga \models E$.

        \item  Let $E$ and $K$ be system of equations with domain $A$ and $g:B \to A$. If $E \implies K$ then $Eg \implies Kg$.

        \item Varieties are strict monomorphisms. In particular, varieties are extremal monomorphisms. 
    \end{enumerate}
\end{proposition}

\begin{proof}
    It is clear.
\end{proof}

\begin{proposition}\label{closed class of equations} Suppose $\Sigma$ is closed under right action by arrows. Then:
    \begin{enumerate}
        \item $\Sigma$-varieties are closed under intersections, products, and are stable under pullbacks.

        \item If $v w$ is a $\Sigma$-variety and $v$ is a monomorphism, then $w$ is also a $\Sigma$-variety. 
    \end{enumerate}
\end{proposition}

\begin{proof}\quad
    \begin{enumerate}
        \item 
        \begin{itemize}
        \item Let $\{v_\gamma\}_{\gamma \in \Gamma}$ be a set of $\Sigma$-varieties with same target. Let $v_\gamma \sdtstile{\text{g.s}}{\Sigma} E_\gamma$, thus, if $v$ is a intersection of $\{v_\gamma\}_{\gamma \in \Gamma}$, then $v \sdtstile{\text{g.s}}{\Sigma} \bigcup_{\gamma \in \Gamma} E_\gamma$.

        \item Now let $\{v_\gamma:V_\gamma \to A_\gamma \}_{\gamma \in \Gamma}$ be a set of $\Sigma$-varieties. Define $V = \prod_{\gamma \in \Gamma} V_\gamma$ with projections $\pi_\gamma: V \to V_\gamma$, and $A = \prod_{\gamma \in \Gamma} A_\gamma$ with projections $\rho_\gamma: A \to A_\gamma$. Now consider $v:V \to A$ such that $\rho_\gamma v = v_\gamma \pi_\gamma$ for all $\gamma \in \Gamma$. We need to show $v$ is a $\Sigma$-variety. If $v_\gamma \sdtstile{\text{g.s}}{\Sigma} E_\gamma$, then $v \sdtstile{\text{g.s}}{\Sigma} \bigcup_{\gamma \in \Gamma} E_\gamma \rho_\gamma$.

        \item Let 
        \[
        \begin{tikzcd}
{} \arrow[rr, "\tilde{v}"] \arrow[dd, "\tilde{f}"] &  & {} \arrow[dd, "f"] \\
                                                   &  &                    \\
{} \arrow[rr, "v"]                                 &  & {}                
\end{tikzcd}
        \]
        be a pullback square, where $v \sdtstile{\text{g.s}}{\Sigma}$. It follows easily that $\tilde{v} \sdtstile{\text{g.s}}{\Sigma} Ef$.
    \end{itemize}
    \item If $v w \sdtstile{\text{g.s}}{\Sigma} E$, then $w \sdtstile{\text{g.s}}{\Sigma} Ev$. 
    \end{enumerate}
\end{proof}

\begin{proposition}\label{Products and varieties} The following statements holds:
    \begin{enumerate}
        \item If $\mathcal{C}$ has equalizers, then varieties are precisely intersections of regular monomorphisms. 

        \item If $\mathcal{C}$ has products or cokernel pairs, then every variety is a regular monomorphism. 

        \item In general, varieties are not closed under composition.
    \end{enumerate}
\end{proposition} 

\begin{proof} \quad
    \begin{enumerate}
        \item If $v \sdtstile{\text{g.s}}{} \{p_\gamma \approx q_\gamma\}_{\gamma \in \Gamma}$, and $v_\gamma$ is an equalizer of $p_\gamma$ and $q_\gamma$, then $v_\gamma \sdtstile{\text{g.s}}{} p_\gamma \approx q_\gamma $ and $v$ is an intersection of $\{v_\gamma\}_{\gamma \in \Gamma}$. 

        \item First suppose that  $\mathcal{C}$ has cokernel pairs. If $v$ is a strict monomorphism and $p,q$ is a cokernel pair of $v$, then $v$ is an equalizer of $p,q$. Therefore, the classes of strict monomorphisms, varieties, and regular monomorphisms are the same. 
        
        Now suppose that $\mathcal{C}$ has products. Let $E=\{p_\gamma,q_\gamma:A \to B_\gamma\}_{\gamma \in \Gamma}$ be a system of equations with domain $A$ and let $v \sdtstile{\text{g.s}}{} E$. Let $B = \prod_{\gamma \in \Gamma} B_\gamma$ with projections $\pi_\gamma : B \to B_\gamma$. Define $p,q:A \to B$ such that $p_\gamma = \pi_\gamma p$ and $q_\gamma = \pi_\gamma q$ for all $\gamma \in \Gamma$. Therefore, $v \sdtstile{\text{g.s}}{} p \approx q$.

        \item In this counterexample we follow the last two paragraphs before to \cite[Proposition 2.1]{kelly_1969}. Consider the category $\mathscr{J}$ of all abelian groups without elements of order $4$. $\mathscr{J}$ is complete, thus every variety is a regular monomorphism. On the other hand, $\mathbf{2}:\mathbb{Z} \to \mathbb{Z}$ is a regular monomorphism but $\mathbf{4}:\mathbb{Z} \to \mathbb{Z}$ is not a regular monomorphism.
    \end{enumerate}
\end{proof}

\begin{definition}\label{generated varieties and equations}
    Let $S$ be a non-empty set of morphisms with target $A$. A system of equations $E$ defined on $A$ is \textbf{generated} by $S$ if $S \models E$, and $S \models K$ implies $E \implies K$. Likewise, a variety $v$ is \textbf{generated} by $S$ if $f \leq v$ for all $f \in S$, and if $w$ is a variety such that $f \leq w$ for all $f \in S$, then $v \leq w$.  
\end{definition}

\begin{definition}\label{complete w.r.t. varieties and equations}
    The category $\mathcal{C}$ is said \textbf{complete w.r.t. varieties} if every non-empty set of morphisms with same target generates some variety. Likewise, $\mathcal{C}$ is said \textbf{complete w.r.t. systems of equations} if every non-empty set of morphisms with same target generates some system of equations.
\end{definition}

\begin{proposition}\label{proposition equations generated}
    $\mathcal{C}$ is complete w.r.t. system of equations if some of the following conditions holds:

        \begin{enumerate}
            \item For every object $A$, the class of equations defined on $A$ is small, up to equivalences. 

            \item $\mathcal{C}$ is complete w.r.t. varieties and has equalizers. 

            \item $\mathcal{C}$ has coproducts and cokernel pairs. 

            \item  $\mathcal{C}$ has cokernel pairs and colimits for diagrams of shape $P$, where $P$ is a poset which has only two minimal elements, and either other element of $P$ is maximal and greater than the two minimal elements of $P$.
            
        \end{enumerate}
    \end{proposition}

    \begin{proof} Let $S=\{f_i : A_i \to A \}_{i \in I}$ be a family of morphisms with target $A$. Then: 
        \begin{enumerate}
            \item  Let $\Sigma$ be the class of all equations $p \approx q$ defined on $A$ such that $S \models p \approx q$. By hypothesis we found a subset $E$ of $\Sigma$ such that every equation in $\Sigma$ is equivalent to some equation in $E$. It follows easily that $E$ is generated by $S$.

            \item Let $v$ be a variety generated by $S$, with $v \sdtstile{\text{g.s}}{} E$. Therefore, it follows by the existence of equalizers that $E$ is generated by $S$.

            \item Let $B = \coprod_{i \in I} A_i$ with coprojections $\mu_i: A_i \to B$. By universal property of $B$ we get $f:B \to A$ such that $f \mu_i = f_i$ for all $i \in I$. Let $p,q$ be a cokernel pair of $f$. Therefore, we have $p \approx q$ is generated by $S$. 

            \item  For each $i \in I$, let $p_i,q_i:A \to B_i$ be a cokernel pair of $f_i$. By hypothesis we found morphisms $p,q:A \to B$, and $r_i:B_i \to B$ such that $r_i p_i = p$ and $r_i q_i = q$ for all $i \in I$, with the following universal property: if $p',q':A \to B'$, and $r'_i:B_i \to B'$ are morphisms such that $r'_i p_i = p'$ and $r'_i q_i = q'$ for all $i \in I$, then there exists a unique morphism $t:B \to B'$ such that $t r_i = r'_i$. Therefore, we have $p \approx q$ is generated by $S$.
        \end{enumerate}
    \end{proof}

    \begin{proposition}\label{proposition varieties generated}
        $\mathcal{C}$ is complete w.r.t. varieties if some of the following conditions holds:

        \begin{enumerate}
            \item $\mathcal{C}$ is wellpowered w.r.t. varieties and has intersections. 

            \item $\mathcal{C}$ is complete w.r.t. systems of equations and every system of equations has a general solution.
        \end{enumerate}
    \end{proposition}

    \begin{proof} Let $S=\{f_i : A_i \to A \}_{i \in I}$ be a family of morphisms with target $A$. Then:
    \begin{enumerate}
        \item Let $\mathcal{V}$ be the class of all varieties $v$ such that $f_i \leq v$ for all $i \in I$. Since $\mathcal{C}$ is wellpowered w.r.t. varieties, we found a subset $\mathcal{V}_0$ of $\mathcal{V}$ such that every variety in $\mathcal{V}$ is isomorphic to some variety in $\mathcal{V}_0$. Let $v$ be a intersection of the morphisms in $\mathcal{V}_0$. Therefore, we have $v$ is a variety generated by $S$.

        \item Let $E$ be a system of equations generated by $S$. Let $v$ be a general solution of $E$. Therefore, we have $v$ is generated by $S$. 
    \end{enumerate}        
    \end{proof}

    {
    \theoremstyle{customthm}
    \newtheorem{mainproblem}[theorem]{Main problem of varieties}
    }
    \begin{mainproblem}\label{mainproblem}
        Given a class $\Sigma$ of equations, to characterize all $\Sigma$-varieties.
    \end{mainproblem}

    {
    \theoremstyle{customthm}
    \newtheorem{inversemainproblem}[theorem]{Inverse main problem of varieties}
    }
    \begin{inversemainproblem}\label{inversemainproblem}
        Given a class $\mathcal{V}$ of varieties, to find a class $\Sigma$ of equations such that every variety in $\mathcal{V}$ is a general solution to some system of $\Sigma$-equations, and every system of $\Sigma$-equations has a general solution in $\mathcal{V}$.
    \end{inversemainproblem}

    It is evident the importance of the main problem, since it means to solve all possible systems of $\Sigma$-equations, and after all, mathematics is pretty much about to solve equations. The above general results give some tools to answer to this problem in particular cases. For example, for varieties of one-sorted algebras, Birkhoff's variety theorem is the best solution in this sense. 

    The inverse main problem is also of great interest. Let us elaborate a little more. Note that there is not necessarily a unique solution to the inverse main problem (for a given class of varieties), and this is not a problem at all! Suppose that $v$ is a general solution of a pair of systems $E_1,E_2$, it could be the case that from $E_1$ we could to find some properties of $v$ which are not as easily to find them as from $E_2$. Therefore, the more we know about which are the systems of equations with $v$ as general solution, the better. Ideally, the finest solution to the inverse main problem is one from which we can get a complete description of the varieties of interest. That is what physics is about, i.e., to find good equations which model (physical) phenomenons. More generally, this corresponds to inverses problems for ordinary or partial differential equations. 

    The class of all $\Sigma$-varieties is closed under isomorphisms. Now let $\mathcal{V}$ to be a non-empty class of varieties. Note that there no exists necessarily a solution to the inverse main problem for $\mathcal{V}$. Indeed, in this case, some regular monomorphism in $\mathcal{V}$ must exists. In particular, for a category with equalizers, there exists some solution (to the inverse main problem for $\mathcal{V}$) if only if there exists some class of regular monomorphism $\mathcal{W}$ such the morphisms in $\mathcal{V}$ are precisely (up to isomorphisms) intersections of morphisms in $\mathcal{W}$. If there exists some solution, then there exists the largest one, just take the union over the class of all these solutions. Moreover, the following proposition characterize the largest solution:
    
    \begin{proposition}\label{unique solution to inverse problem}
        Let $\mathcal{V}$ be a class of varieties. We have that any two solutions to the inverse main problem for $\mathcal{V}$ are equivalents. Therefore, a solution to the inverse main problem for $\mathcal{V}$ is the largest one if only if it is closed under equivalences. 
    \end{proposition}

    \begin{proof}
        It follows straightforward from the definitions. 
    \end{proof}
    
    We should not to expect to solve both the main problem and its inverse in their utmost generality. Each context has its particular difficulties. As a general rule, we should give a correct setting in which to find proper answers to these type of questions, what we mean is to find a correct choice of the morphisms between the objects and varieties of interest. In the following section we show some examples which illustrate the concepts introduced before. A more elaborate example is given in Section \ref{main section} for subvarieties of algebraic categories.

    To end this section, we consider the duals of some of the previous concepts which we are going to use in Section \ref{main section} w.r.t. Lawvere covarieties. Of course, all propositions given above has its respective dual. For example, the dual of Proposition \ref{Products and varieties} (ii) tell us that covarieties are regular epimorphisms if $\mathcal{C}$ has coproducts or kernel pairs. Note that the concept of equation is self-dual.

    \begin{definition}
        A \textbf{cosystem of equations} is a non-empty set of equations with same target.
    \end{definition}

    \begin{definition}
        Let $E$ be a cosystem of equations with target $A$. A \textbf{cosolution} of (the equations in) $E$ is an arrow $a$ with source $A$ such that $a f = a g $ for all equations $f \approx g$ in $E$.
    \end{definition}

    \begin{definition}\label{covariety definition}
            A \textbf{general cosolution} of a cosystem of equations $E$ is a cosolution $v$ of $E$ with the property that for each cosolution $a$ of $E$ there exists a unique morphism $\tilde{a}$ such that $a = \Tilde{a} v$. A \textbf{covariety (morphism)} is a general cosolution of some cosystem of equations. 
    \end{definition}

    Similarly, we order covarieties and cosystems of equations, and define a left action of arrows over cosystems of equations. In any case, the context makes clear in which sense we are comparing arrows, varieties, covarieties, and so on. 

    \section{Examples}\label{examples}

    \begin{example}
    
    Consider an open interval of real numbers $(a,b)$ and let $\textbf{Functional Spaces}(a,b)$ be the category of real vectorial spaces of continuous functions from $(a,b)$ to $\mathbb{R}$, with linear transformations as morphisms.  
    
    We are going to consider linear differential homogeneous equations of order $n$ defined on $C^{(n)}(a,b)$, i.e., equations $L \approx 0$, where $L:C^{(n)}(a,b) \to C(a,b)$ is a linear operator of the form $L(y) = y^{(n)}+p_1 y^{(n-1)} + \cdots + p_n y$, with $p_1,\ldots,p_n$ continuous. 

    The Wronskian criterion determines the linear subspaces of $C^{(n)}(a,b)$ which are general solutions to some linear differential homogeneous equation of order $n$. Given $n$ functions $f_1,\ldots,f_n$ in $C^{(n)}(a,b)$, define their Wronskian as
    \[
        W(f_1,\ldots,f_n) = 
            \begin{vmatrix}
                f_1 & f_2 & \cdots & f_n \\
                f_1 ' & f_2 ' & \cdots & f_n ' \\
                \vdots & \vdots & \ddots & \vdots \\
                f_1^{(n-1)} & f_2^{(n-1)} & \cdots & f_n^{(n-1)} 
            \end{vmatrix}
    \]
    
    Let $V$ be a linear subspace of $C^{(n)}(a,b)$ of dimension $n$. Note that if there is some basis $\{y_1,\ldots,y_n\}$ of $V$ such that $W(y_1,\ldots,y_n)$ is never zero on $(a,b)$, then the Wronskian of any basis of $V$ is never zero on $(a,b)$. In this case we have that every linear transformation $T:V \to C(a,b)$ is of the form $T(y) = y^{(n)}+p_1 y^{(n-1)} + \cdots + p_n y$ for all $y \in V$, for some $p_1,\ldots,p_n$ continuous, moreover, this representation of $T$ is unique. In particular, this shows that the inclusion transformation from $V$ into $C^{(n)}(a,b)$ is a general solution to some linear homogeneous equation of order $n$. The reciprocal of the last implication is also true, i.e., if the inclusion transformation from $V$ into $C^{(n)}(a,b)$ is a general solution to some (actually, a unique) linear homogeneous equation of order $n$, then the Wronskian of any basis of $V$ is never zero on $(a,b)$. 
    \end{example}

    \begin{example}\label{forma differential varieties}
        Let $\mathcal{R}$ be a subcategory of $Set$ with only one object $R$. A $\mathcal{R}$-variety is a subset $V$ of $R$ such there exists a family $\{f_\gamma,g_\gamma:R \to R\}_{\gamma \in \Gamma}$ of pairs of $\mathcal{R}$-morphisms such that $$V = \{ x \in R \mid f_\gamma(x) = g_\gamma(x) \; \text{for all} \; \gamma \in \Gamma \}.$$
        We denote by $\textbf{Var}(\mathcal{R})$ to the category of $\mathcal{R}$-varieties. A morphism in $\textbf{Var}(\mathcal{R})$ from $V$ to $W$ is a function $f:V \to W$ which is the restriction of some $\mathcal{R}$-morphism $\varphi:R \to R$, i.e., the following diagram commutes
        \[
        \begin{tikzcd}
            R \arrow[rr, "\varphi"]                     &  & R                  \\
                                         &  &                    \\
            V \arrow[rr, "f"] \arrow[uu, hook] &  & W \arrow[uu, hook]
        \end{tikzcd}
        \] 

    Observe that $\mathcal{R}$ and $\textbf{Var}(\mathcal{R})$ are concrete categories over $Set$, and there is a concrete functor from $\mathcal{R}$ to $\textbf{Var}(\mathcal{R})$. Also, every inclusion between $\mathcal{R}$-varieties is a variety in $\textbf{Var}(\mathcal{R})$, and every system of equations in $\textbf{Var}(\mathcal{R})$ has a general solution in $\textbf{Var}(\mathcal{R})$, which is an inclusion of $\mathcal{R}$-varieties.
    
    In particular, let $R$ be an integral domain with a derivative map, i.e., for each $f \in R$, the derivative of $f$, denoted by $f'$, is an element of $R$ and this map satisfies Leibniz's product rule $(fg)' = f'g+fg'$. Let $F$ be the quotient field of $R$, therefore, there is a unique extension of the derivative map from $R$ to $F$. An element $r \in F$ is called a constant if $r' = 0$, thus $0$ and $1$ are constants. We assume that every constant belongs to $R$.  

    We define the set of differential operator over $R$ as the set of functions over $R$ generated under sums, products, and compositions by constant maps, the identity map, and the derivative map. Note that differential operators over $R$ are closed under composition, and the identity map on $R$ is a differential operator. Therefore, $R$ with the differential operators define a category $\mathcal{R}$, which is a subcategory of $Set$. 

    For each $n \geq 1$ let $T_n$ be a differential operator over $R$. Define $V_n \subset R$ such that $f \in V_n$ if there exist constants $a_1,a_2,\ldots,a_n \in R$ not all zeros such that $a_1 T_1(f) + a_2 T_{2}(f)+ \cdots + a_n T_n(f) = 0 .$ We have $V_n$ is $\mathcal{R}$-variety, in fact,
    $V_n = \{ f \in R \mid W(T_1(f),T_2(f),\ldots,T_n(f)) = 0 \}$, where $W$ is the Wronskian map. 
        
    As a corollary of the above we have the following non-trivial fact
    \[  W(T_1(f),T_2(f),\ldots,T_n(f)) = 0
        \quad \implies \quad   W(T_1(f),T_2(f),\ldots,T_n(f),T_{n+1}(f))=0 .
    \]

    There is two main applications of this example:
        
    \begin{itemize}
        \item Let $p$ be a complex number and define $F$ as the set of all germs at $p$ defined by functions $\mathbb{C} \to \mathbb{C}$ holomorphic at $p$. Recall the germ at $p$ of a function $f:\mathbb{C} \to \mathbb{C}$ is the equivalence class of all functions $g:\mathbb{C} \to \mathbb{C}$ which are locally equal to $f$ at $p$. It is clear $F$ is a field under usual sums and products, with complex differentiation as a derivative map on $F$. The varieties defined in this context account as local solutions to differential equations at $p$, for complex value functions of complex variable. In particular, a function $f:\mathbb{C} \to \mathbb{C}$ is a solution of some linear homogeneous differential equation with constant coefficients of order at most $n$ in some neighbourhood of $p$, i.e., there exist $a_0,a_1,\ldots,a_n \in \mathbb{C}$ not all zeros such that $a_0 f^{(n)} + a_1 f^{(n-1)} + \cdots + a_{n-1} f'  +a_n f$ vanishes in some neighbourhood of $p$ if only if $f$ is holomorphic at $p$ and $W\big(f,f',\ldots,f^{(n)}\big)$ is locally zero at $p$. 
            
        \item Let $R$ be an integral domain and consider the ring $R[[x]]$ of formal power series with coefficients in $R$. It is clear $R[[x]]$ is an integral domain, and we have a derivative map on $R[[x]]$: $(\sum_{n \geq 0} a_n x^n )' = \sum_{n \geq 0} (n+1)a_{n+1}x^n$. Note that all constants belongs to $R[[x]]$. The varieties defined in this context account as solutions of some recurrence equations. In particular, we can define linear recurrence equations in this context, in fact, for $a_0,a_1,\ldots,a_n \in R$ not all zeros we have $f \in R[[x]]$ satisfies $a_0 f_k + a_1 f_{k+1} + \cdots + a_n f_{k+n} = 0$ for all $k \geq 0$ if only if $\left((a_0 x^n + a_1 x^{n-1} + \cdots + a_n)f \right)^{(n)} = 0$. Therefore, $f \in R[[x]]$ is a solution of some linear recurrence equation of order at most $n$ if only if $W\big(f^{(n)},(xf)^{(n)},\ldots,(x^n f)^{(n)}\big) = 0$. 
        \end{itemize}
    \end{example}

    \begin{example}\label{affine varieties}    
        Consider the category $\textbf{Affine Varieties}$ whose objects are the common zero locus of some finite set of polynomials with complex coefficients, with regular maps as morphisms. 
    
        Let $\mathbb{A}^n$ be the affine complex space of dimension $n$. A polynomial equation defined on $\mathbb{A}^n$ is an equation $p \approx q$, where $p,q$ are polynomial maps defined on $\mathbb{A}^n$. We have that every inclusion map into $\mathbb{A}^n$ is a general solution to some finite system of polynomial equations defined on $\mathbb{A}^n$, and every system of polynomial equations defined on $\mathbb{A}^n$ has a general solution, which it is an inclusion map into $\mathbb{A}^n$.
    \end{example}

    \begin{example}\label{projective varieties}
        Let $\textbf{Projective Varieties}$ be the category whose objects are the common zero locus of some finite set of homogeneous polynomials with complex coefficients, with regular maps as morphisms. 
    
        Let $\mathbb{P}^n$ be the projective complex space of dimension $n$. A polynomial equation defined on $\mathbb{P}^n$ is an equation $p \approx q$, where $p,q$ are polynomial maps defined on $\mathbb{P}^n$. We have that every inclusion map into $\mathbb{P}^n$ is a general solution to some finite system of polynomial equations defined on $\mathbb{P}^n$, and every system of polynomial equations defined on $\mathbb{P}^n$ has a general solution, which is an inclusion map into $\mathbb{P}^n$. 
    \end{example}

    \begin{example}
        Consider the category \textbf{CMet} of complete metric spaces with non-expansive maps as morphisms. Let $X$ be a complete metric space and consider the equation $T \approx 1_X$ defined on $X$, where $T:X \to X$ is a non-expansive map. By Banach's fixed point theorem we have that this equation has exactly only one solution. Therefore, if $E$ is the trivial metric space with only one point, then there exists a unique morphism $F:E \to X$ such that $TF = 1_XF$, which it is a general solution of the equation $T \approx 1_X$. 

        On the other hand, note that $E$ is a final object of \textbf{CMet}, and any map $E \to X$ is a general solution of some equation $T \approx 1_X$ for some non-expansive map $T:X \to X$ (trivially, a constant map). 
    \end{example}

    \begin{example}
        Consider the category \textbf{Grp} of groups and homomorphisms. Let $G$ be a group. If $N$ is a normal subgroup of $G$, then the inclusion $N \hookrightarrow G$ is a general solution to the equation $\pi \approx 0$, where $\pi:G \to G/N$ is canonical. Next, for a non-empty subset $S$ of $G$, the inclusion $C_G(S) \hookrightarrow G$ of the centralizer $C_G(S)$ of $S$ in $G$, is a general solution of $\{ \varphi_g \approx 1_G \}_{g \in S}$, where $\varphi_g:G \to G$ is the inner automorphism of $G$ defined by $g$. Similarly, let $G^{\operatorname{ab}}$ be the abelianization of $G$, thus, the canonical homomorphism $G \to G^{\operatorname{ab}}$ is a general cosolution of $\{\varphi_g \approx 1_G \}_{g \in G}$. Observe that both the center and the abelianization of $G$ are defined by the same equations $\{\varphi_g \approx 1_G \}_{g \in G}$, the former as a general solution and the latter as a general cosolution. 

        In these examples we have shown the canonical equations which define normal subgroups, centralizers, and abelianizations. However, it is well-known that every monomorphism (of groups) is a regular monomorphism, and every monomorphism is, up to isomorphism, an inclusion homomorphism. Analogously, the epimorphisms of groups are regular epimorphisms (actually, surjective homomorphisms). 
    \end{example}

    \begin{example}
        In \textbf{Cat}, a functor is a monomorphism if only if it is an embedding, a full embedding functor is a regular monomorphism, and regular monomorphisms are conservative embeddings.  

        It is very natural how varieties are defined in \textbf{Cat}. Explicitly, let $\{ P_\gamma,Q_\gamma : \mathcal{A} \to \mathcal{B}_\gamma \}_{ \gamma \in \Gamma}$ be a family of parallel functors with same domain. Define the subcategory $\mathcal{V}$ of $\mathcal{A}$ such that an object $x$ of $\mathcal{A}$ belongs to $\mathcal{V}$ if $P_\gamma x = Q_\gamma x$ for all $\gamma \in \Gamma$, and a morphism $f:x \to y$ in $\mathcal{A}$ between objects of $\mathcal{V}$ belongs to $\mathcal{V}$ if $P_\gamma f = Q_\gamma f$ for all $\gamma \in \Gamma$. It is clear that the inclusion $\mathcal{V} \hookrightarrow \mathcal{A}$ is a general solution of $\{ P_\gamma \approx Q_\gamma \}_{ \gamma \in \Gamma}$. 
    \end{example}

    \section{Inserters in algebraic categories}\label{inserters section}

In this section we consider only locally small categories. Let $F,G:\mathcal{A} \to \mathcal{B}$ be a pair of parallel functors. 

\begin{definition}\label{definition inserters}
    An \textbf{inserter} from $F$ to $G$ is a category $\mathcal{V}$, with a functor $U:\mathcal{V} \to \mathcal{A}$ and a natural transformation $\lambda:FU \to GU$ with the property that for each functor $V:\mathcal{D} \to \mathcal{A}$ and natural transformation $\alpha:FV \to GV$ there exists a unique functor $W:\mathcal{D} \to \mathcal{V}$ such that $V=UW$ and $\alpha = \lambda W$. 
\end{definition}

Concretely, let $\textbf{Ins}(F,G)$ be the category whose objects are pairs $(A,r)$, with $A$ an object of $\mathcal{A}$ and $r:F(A) \to G(A)$. A morphism $d:(A,r) \to (B,s)$ in $\textbf{Ins}(F,G)$ is a morphism $d:A \to B$ in $\mathcal{A}$ such that the diagram
\[
\begin{tikzcd}
FA \arrow[rr, "Fd"] \arrow[dd, "r"] &  & FB \arrow[dd, "s"] \\
                                        &  &                      \\
GA \arrow[rr, "Gd"]                 &  & GB                
\end{tikzcd}
\]
commutes. Define the forgetful functor $U:\textbf{Ins}(F,G) \to \mathcal{A}$ by
\[
U \left( (A,r) \xrightarrow{d} (B,s) \right) = A \xrightarrow{d} B  
\]
and the inserted transformation $\lambda:FU \to GU$ by  $\lambda(A,r) = r$. We have $\textbf{Ins}(F,G)$, with $U$ and $\lambda$, is an inserter from $F$ to $G$. 

Clearly, the forgetful functor $U$ is faithful, this makes $\textbf{Ins}(F,G)$ a concrete category over $\mathcal{A}$. Also it is clear that there is only one inserter from $F$ to $G$, up to concrete isomorphism over $\mathcal{A}$. The forgetful functor $U$ has many others properties, for example, $U$ is uniquely transportable (in particular, $U$ is amnestic) and conservative. 

The above is valid in general. Now, if $\mathcal{A}$ has all limits of shape $\mathcal{J}$ and $G$ preserves them, then $\textbf{Ins}(F,G)$ also has all limits of shape $\mathcal{J}$ and $U$ preserves them. Similarly, if $\mathcal{A}$ has all colimits of shape $\mathcal{J}$ and $F$ preserves them, then $\textbf{Ins}(F,G)$ also has all colimits of shape $\mathcal{J}$ and $U$ preserves them. These facts about inserters categories are well-known, and they are easy to prove. 

\begin{theorem}\label{insertes algebras and coalgebras}
    If $H$ is a left adjoint to $G$, then $\textbf{Ins}(F,G)$ and $\textbf{Ins}(HF,1_\mathcal{A})$ are isomorphic as concrete categories over $\mathcal{A}$. Similarly, if $H$ is a right adjoint to $F$, then $\textbf{Ins}(F,G)$ and $\textbf{Ins}(1_\mathcal{A},HG)$ are isomorphic as concrete categories over $\mathcal{A}$. 
\end{theorem}

\begin{proof}
    Let $H$ be a left adjoint to $G$, with unit $\eta$ and counit $\varepsilon$. Consider $\textbf{Ins}(HF,1_\mathcal{A})$ with forgetful functor $V$. Define functors $\Psi:\textbf{Ins}(F,G) \to \textbf{Ins}(HF,1_\mathcal{A})$ and $\Phi: \textbf{Ins}(HF,1_\mathcal{A}) \to \textbf{Ins}(F,G)$ by

    \[
    \Psi \left((A,r) \xrightarrow{d} (B,s) \right) = (A,\varepsilon_A H(r)) \xrightarrow{d} (B,\varepsilon_B H(s)),
    \]
    and
    \[
    \Phi \left((A,r) \xrightarrow{d} (B,s) \right) = (A,G(r) \eta_{FA}) \xrightarrow{d} (B,G(s) \eta_{FB}).
    \]
    It is clear that $\Psi$ and $\Phi$ are concrete functor over $\mathcal{A}$, i.e., $U\Phi = V$ and $V \Psi = U$. Also, it follows easily by triangular identities of the adjunction $(\eta, \varepsilon) :H \vdash G$ that $\Psi$ and $\Phi$ are inverses of each other.

    The proof for the case in which $H$ is a right adjoint to $F$ is quite similar.
\end{proof}

\begin{theorem}\label{TeoremaInserters} 
    Suppose $\mathcal{A}$ is cocomplete, $F$ is finitary, and $G$ has a left adjoint. Then $U$ has a left adjoint. Similarly, if $\mathcal{A}$ is algebraic, $F$ preserves sifted colimits, and $G$ has a left adjoint, then $\textbf{Ins}(F,G)$ and $U$ are algebraics.
\end{theorem}

\begin{proof} 
        Let $H$ be a left adjoint to $G$, with unit $\eta$ and counit $\varepsilon$. Consider $\textbf{Ins}(HF,1_\mathcal{A})$ with forgetful functor $V$. By a generalization of \cite[Chapter 12]{adámek_rosický_vitale_lawvere_2010} we have $V$ has a left adjoint since $\mathcal{A}$ is cocomplete and $HF$ is finitary because $F$ is finitary (by assumption) and $H$ is finitary too (it preserves all colimits). 
        
        Let $K$ be a left adjoint to $V$, with unit $\theta$ and counit $\tau$. Consider $\Phi$ and $\Psi$ defined in the first paragraph of the proof of Theorem \ref{insertes algebras and coalgebras}. Therefore, we have $\Phi K$ is a left adjoint to $U$, with unit $\theta$ and counit $\Phi \tau \Psi$.

        Now assume that $\mathcal{A}$ is algebraic, $F$ preserves sifted colimits, and $G$ has a left adjoint. Since $\mathcal{A}$ is cocomplete and $F$ preserves sifted colimits, we have $U$ preserves sifted colimits. Therefore, it remains to prove $\textbf{Ins}(F,G)$ is algebraic. 

        Let $L$ be a left adjoint to $U$ (like the functor $\Phi K$ given above). We have $L$ preserves strong generators and perfectly presentable objects, this is easy to prove. Therefore, $\textbf{Ins}(F,G)$ has a strong generator consisting of perfectly presentable objects since $\mathcal{A}$ is algebraic. Finally, we will conclude this proof if we prove $\textbf{Ins}(F,G)$ is cocomplete. 

        Once again, by a generalization of  \cite[Chapter 12]{adámek_rosický_vitale_lawvere_2010} we have $\textbf{Ins}(HF,1_\mathcal{A})$ is cocomplete since $\mathcal{A}$ is cocomplete and $HF$ preserves sifted colimits. Thus, $\textbf{Ins}(F,G)$ is cocomplete because it is isomorphic to $\textbf{Ins}(HF,1_\mathcal{A})$. 
\end{proof}

As an application of the results in this section we have the following:

\begin{corollary}
    The category $\Sigma$-Alg of $\Sigma$-algebras, for an $S$-sorted signature $\Sigma$, is an algebraic inserted category.
\end{corollary}

\begin{proof}
    We refer the reader to \cite[Example 1.5 and Example 1.10]{adámek_rosický_vitale_lawvere_2010} for a definition of the category $\Sigma$-Alg of $\Sigma$-algebras. These examples show that $Set^S$ is an algebraic category equivalent to $Alg \, S^*$. Explicitly, the functor $Set^S \to Alg \, S^*$ defined on objects by
    \[
    A=\langle A_s \rangle_{s \in S} \quad \longmapsto \quad A^* =\langle A_{s_0} \times \cdots \times A_{s_{n-1}} \rangle_{s_0 \ldots s_{n-1} \in S^*} ,
    \]
    is an equivalence. 

     The arity function gives two functions $s: \Sigma \to S^*$ and $t:\Sigma \to S$ by composition with the projections from $S^* \times S$. These functions naturally define functors $s':Set^{S^*} \to Set^\Sigma$ and $\hat{t}:Set^S \to Set^\Sigma$ by composition with $s$ and $t$ respectively. Define the functor $\hat{s}:Set^S \to Set^\Sigma$ as the composition
    \[
    \begin{tikzcd}
Set^S \arrow[dd] \arrow[rr, "\hat{s}"] &  & Set^\Sigma                 \\
                                       &  &                            \\
{Alg \, S^*} \arrow[rr, hook]          &  & Set^{S^*} \arrow[uu, "s'"]
\end{tikzcd}
    \]
    It is clearly seen that the category $\Sigma$-Alg is exactly the category $\textbf{Ins}(\hat{s},\hat{t})$. 

    It remains to prove $\textbf{Ins}(\hat{s},\hat{t})$ is algebraic. Thus, it is sufficient to show that $\hat{s}$ preserves sifted colimits and $\hat{t}$ has a left adjoint, see Theorem \ref{TeoremaInserters}. In first place, observe that $Set^\Sigma$, $Set^S$ and $Set^{S^*}$ has limits and sifted colimits and these are computed objectwise, this implies that $s'$ and $\hat{t}$ preserves these limits and colimits. On the other hand, the inclusion $Alg\, S^* \hookrightarrow Set^{S^*}$ also preserves these limits and colimits, see \cite[Proposition 1.21 and Proposition 2.5]{adámek_rosický_vitale_lawvere_2010}, thus $\hat{s}$ preserves sifted colimits since it is a composition of sifted colimit preserving functors. As $S$ is a small discrete category and $Set$ has a cogenerator family, then $Set^S$ also has a cogenerator family. In the same way we have $Set^S$ is wellpowered. Therefore, it follows by Special adjoint functor theorem \cite[Theorem 3.3.4]{borceux_1994V1} that $\hat{t}$ has a left adjoint. 
\end{proof}

\section{Birkhoff varieties and Lawvere covarieties}\label{main section}

Let us take the Birkhoff's variety theorem as a definition of subvarieties of algebraic categories. 

\begin{definition}
    Let $\mathcal{A}$ be an algebraic category. A \textbf{subvariety} of $\mathcal{A}$ is a full subcategory $\mathcal{V}$ of $\mathcal{A}$ closed under products, subalgebras, regular quotients, and direct unions. 
\end{definition}

\begin{remark}
Subvarieties of algebraic categories, and their inclusions functors, are algebraic. This follows by  \cite[Corollary 10.15 and Corollary 10.17]{adámek_rosický_vitale_lawvere_2010}. 
\end{remark}

\begin{definition}\label{birkhoff equation}
    A \textbf{Birkhoff equation} is an equation $P \approx Q$ in \textbf{AlgCat} such that there exists an algebraic, faithful, conservative, and amnestic functor $U$ such that $UP = UQ$.
\end{definition}

Observe that the class of Birkhoff equations is closed under right action by algebraic functors, therefore the results of Proposition \ref{closed class of equations} are valid in this context. By Theorem \ref{main theorem} below, we have that the class of Birkhoff equations is a solution to the inverse main problem of varieties for the class of inclusions of varieties of algebraic categories. 

Note how the conditions on $U$ for a Birkhoff equation $P \approx Q$ correspond to several properties of forgetful functors of inserters over algebraic categories. As it can be seen in the proof of Theorem \ref{main theorem}, we could take $U$ like these forgetful functors. However, the conditions on $U$ given in Definition \ref{birkhoff equation} are sufficient to get Theorem \ref{main theorem}. 

Another fact to notice in the following proof is how the varieties of an algebraic category $\mathcal{A}$ are close related to inserters and equifiers, where this relation is given by a set of perfectly presentable objects of $\mathcal{A}$ which generates to $\mathcal{A}$ under sifted colimits. 

\begin{theorem}\label{main theorem}
    Every system of Birkhoff equations has a general solution, which it is an inclusion functor $\mathcal{V} \hookrightarrow \mathcal{A}$, where $\mathcal{V}$ is a subvariety of $\mathcal{A}$. Reciprocally, if $\mathcal{V}$ is a subvariety of $\mathcal{A}$ then the inclusion functor $\mathcal{V} \hookrightarrow \mathcal{A}$ is a general solution of some system of Birkhoff equations.
\end{theorem}

\begin{proof}
        Let $E$ be a system of Birkhoff equations defined on an algebraic category $\mathcal{A}$. Define $\mathcal{V}$ as the full subcategory of $\mathcal{A}$ of objects $x$ such that $P(x) = Q(x)$ for all equations $P \approx Q$ in $E$. We have $\mathcal{V} \hookrightarrow \mathcal{A} \sdtstile{\text{g.s}}{} E$ in \textbf{Cat}. So, we only need to show $\mathcal{V}$ is a subvariety of $\mathcal{A}$. It is easy to show $\mathcal{V}$ is closed under limits and sifted colimits, in particular, $\mathcal{V}$ is closed under products and direct unions. Moreover, the inclusion $\mathcal{V} \hookrightarrow \mathcal{A}$ preserves these limits and colimits. 

        We are going to prove that $\mathcal{V}$ is closed under subalgebras and regular quotients. Let $P \approx Q$ be an equation in $E$, and let $U$ be an algebraic, faithful, conservative, and amnestic functor such that $UP=UQ$. 
        
        Suppose that $m:x \to y$ is a monomorphism in $\mathcal{A}$, and $y$ is an object of $\mathcal{V}$. Now consider a pullback square
        \[
        \begin{tikzcd}
{} \arrow[dd, "u"'] \arrow[rr, "v"] &  & {} \arrow[dd, "Q(m)"] \\
                                    &  &                     \\
{} \arrow[rr, "P(m)"']                &  & {}                 
\end{tikzcd}
        \]
        
        Let $g=UP(m) = UQ(m)$. As $U$ is algebraic, then $U$ preserves pullbacks, and $g$ is a monomorphism since $m$ is a monomorphism too, and $UP$ preserves monics. Thus we have a pair of pullback squares
        \[
        \begin{tikzcd}
{} \arrow[dd, "U(u)"'] \arrow[rr, "U(v)"] &  & {} \arrow[dd, "g"] &  & {} \arrow[dd, "1_{UP(x)}"'] \arrow[rr, "1_{UP(x)}"] &  & {} \arrow[dd, "g"] \\
                                          &  &                    &  &                                                     &  &                    \\
{} \arrow[rr, "g"]                       &  & {}                 &  & {} \arrow[rr, "g"]                                  &  & {}                
\end{tikzcd}
        \]
        Then we find an isomorphism $h$ such that $hU(u) = hU(v) = 1_{UP(x)}$, it follows that $U(u) = U(v)$ is an isomorphism, then $u$ and $v$ are isomorphisms too since $U$ is conservative. Now ${vu^{-1}:Px \to Qx}$ satisfies $U(vu^{-1}) = 1_{UP(x)}$, thus $vu^{-1}$ is an identity since $U$ is amnestic and $vu^{-1}$ is an isomorphism. Therefore, $Px = Qx$ and $\mathcal{V}$ is closed under subalgebras.

        On the other hand, suppose that $e:x \to y$ is a coequalizer in $\mathcal{A}$ of $r,s:z \to x$, where $x$ is an object of $\mathcal{V}$. Both $P$ and $Q$ preserves coequalizers, so we have $P(e)$ and $Q(e)$ are coequalizers of the same pair of arrows, thus we get an isomorphism $h$ such that $h P(e) = Q(e)$, then $U(h) UP(e)  = UQ(e)$, since $UP(e) = UQ(e)$ is a regular epimorphism, then $U(h)$ is an identity, thus $h$ is an identity. Therefore, $Py = Qy$ and $\mathcal{V}$ is closed under regular quotients.  

        
        Conversely, suppose $\mathcal{V}$ is a subvariety of $\mathcal{A}$. By \cite[Corollary 10.24] {adámek_rosický_vitale_lawvere_2010} $\mathcal{V}$ is a regular epireflective subcategory of $\mathcal{A}$, thus we have an epireflector $R:\mathcal{A} \to \mathcal{V}$ of $\mathcal{V} \hookrightarrow \mathcal{A}$, with regular epireflections $r_x:x \to Rx$.

        Let $\mathcal{G}$ be a set of perfectly presentable objects of $\mathcal{A}$ such that every object of $\mathcal{A}$ is a sifted colimit of objects of $\mathcal{G}$. For every $x \in \mathcal{G}$ suppose that $r_x$ is a coequalizer of ${p_x,q_x:S_x \to x}$. Define $\mathcal{V}_x = \textbf{Ins}(\mathcal{A}(x,-),\mathcal{A}(S_x, -))$, with $U_x:\mathcal{V}_x \to \mathcal{A}$ and $\lambda_x : \mathcal{A}(x,U_x - ) \to \mathcal{A}(S_x,U -)$. Now we define functors $P_x,Q_x : \mathcal{A} \to \mathcal{V}_x$ such that $\lambda_x P_x = \mathcal{A}(p_x,-)$, $\lambda_x Q_x = \mathcal{A}(q_x,-)$, and ${U_x P_x =  U_x Q_x = 1_\mathcal{A}}$. 

        We are going to show that $\mathcal{V} \hookrightarrow \mathcal{A}$ is a general solution of $\{P_x \approx Q_x \mid x \in \mathcal{G} \}$. Note that $\mathcal{A}(x,-)$ preserves sifted colimits since $x$ is perfectly presentable (by definition). On the other hand, every hom functor defined on $\mathcal{A}$ has a left adjoint since $\mathcal{A}$ is cocomplete. Then, $\mathcal{V}_x$ and $U_x$ are algebraics by Theorem \ref{TeoremaInserters}. Since $U_x$ is a forgetful functor of an inserted category, we have $U_x$ is faithful, conservative, and amnestic. Also note that $P_x$ and $Q_x$ are algebraics since $U_x P_x = U_x Q_x = 1_\mathcal{A}$ and $U_x$ reflects limits and sifted colimits (the latter follows because $U_x$ preserves limits and sifted colimits, $\mathcal{V}_x$ is algebraic, and $U_x$ is conservative). Therefore, we have proved that every $P_x \approx Q_x$ is a Birkhoff equation. 

        Define $\mathcal{V}'$ as the subvariety of $\mathcal{A}$ defined by $\{P_x \approx Q_x \mid x \in \mathcal{G} \}$. Observe that an object $y$ of $\mathcal{A}$ belongs to $V'$ if only if $\mathcal{A}(p_x,y) = \mathcal{A}(q_x,y)$ for all $x \in \mathcal{G}$. In other terms, $\mathcal{V}'$ is the intersection of the equifiers defined by the pairs of natural transformations $\mathcal{A}(p_x,-)$ and $\mathcal{A}(q_x,-)$, with $x \in \mathcal{G}$. 
        
        Let us see that $\mathcal{V} \subset \mathcal{V}'$. If $y \in \mathcal{V}$, to show that $y \in \mathcal{V}'$ we must to prove that $\mathcal{A}(p_x,y) = \mathcal{A}(q_x,y)$ for all $x \in \mathcal{G}$. Now, if $x \in \mathcal{G}$, then $\mathcal{A}(p_x,y) = \mathcal{A}(q_x,y)$ if $f p_x = f q_x$ for all $f:x \to y$. By universal property of $r_x:x \to Rx$ we get a unique $g:Rx \to y$ such that $f= g r_x$ because $y \in \mathcal{V}$. Thus, $f p_x = g r_x p_x = g r_x q_x = f q_x$. Therefore, $y \in \mathcal{V}'$ and $\mathcal{V} \subset \mathcal{V}'$. 

        Now we want to show $\mathcal{V}' \subset \mathcal{V}$. Let $R':\mathcal{A} \to \mathcal{V}'$ be an epireflector of $\mathcal{V}' \hookrightarrow \mathcal{A}$, with regular epireflections $r'_x:x \to R'x$. For each $x \in \mathcal{G}$, we have $R'x \in V'$, hence $R'x$ satisfies $\mathcal{A}(p_x,R'x) = \mathcal{A}(q_x,R'x)$. Thus $r'_x p_x = r'_x q_x$, and we get a unique $t_x:Rx \to R'x$ such that $r'_x = t_x r_x$. Since $r'_x$ is regular epi, then $t_x$ is also a regular epi. Therefore, $R'x$ is a regular quotient of $Rx \in \mathcal{V}$, and $R'x \in \mathcal{V}$ since $\mathcal{V}$ is closed under regular quotients. 

        For $y \in \mathcal{V}'$ let $\{\mu_i : x_i \to y \}_{i \in D}$ be a sifted colimit cocone such that $x_i \in \mathcal{G}$ for all  $i \in D$. Since $R'$ is a left adjoint, then $R'$ preserves colimits, thus $\{ R' \mu_i : R'x_i \to R'y \}_{i \in D}$ is a sifted colimit cocone. We have $R'x_i \in \mathcal{V}$ for all $i$ since $x_i \in \mathcal{G}$, thus $\{ R' \mu_i : R'x_i \to R'y \}_{i \in D}$ is a sifted colimit cocone in $\mathcal{A}$ of objects in $\mathcal{V}$. Since $\mathcal{V}$ is closed under sifted colimits, we have $R'y$ belongs to $\mathcal{V}$, as $y$ and $R'y$ are isomorphic in $\mathcal{A}$ because $y \in \mathcal{V}'$, then $y \in \mathcal{V}$. Therefore, we have concluded that $\mathcal{V}' \subset \mathcal{V}$, $\mathcal{V} = \mathcal{V}'$. 
        \end{proof}

        \begin{theorem}\label{regularmonobirkhoff}
    Every Birkhoff variety is a coreflexive equalizer of some coreflexive Birkhoff equation. 
\end{theorem}

\begin{proof}
    Recall that the product category of a family of algebraic categories, with the respective projections, are algebraics. In other words, \textbf{AlgCat} is closed under products as a subcategory of \textbf{Cat}. 

    Let $v:\mathcal{B} \to \mathcal{A}$ be a Birkhoff variety. By Theorem \ref{main theorem} $v$ is isomorphic to $\mathcal{V} \hookrightarrow \mathcal{A}$ for some subvariety $\mathcal{V}$ of $\mathcal{A}$. Following the proof of Theorem \ref{main theorem} we found a system of coreflexive Birkhoff equations $\{P_x, Q_x: \mathcal{A} \to \mathcal{V}_x \}_{x \in \mathcal{G}}$, with $\mathcal{V} \hookrightarrow \mathcal{A}$ as a general solution, and a family of algebraic, faithful, conservative, and amnestic functors $\{U_x : \mathcal{V}_x \to \mathcal{A}\}_{x \in \mathcal{G}}$ such that $U_x P_x = U_x Q_x = 1_\mathcal{A}$. 
    
    Now are going to almost verbatim the proof of Proposition \ref{Products and varieties} (ii). Define $\mathcal{B} = \prod_{x \in \mathcal{G}} \mathcal{V}_x$ with projections $\pi_x:\mathcal{B} \to \mathcal{V}_x$, and let $P,Q:\mathcal{A} \to \mathcal{B}$ be defined by $P_x = \pi_x P$ and $Q_x = \pi_x Q$ for all $x \in \mathcal{G}$. Define $\mathcal{C} = \prod_{x \in \mathcal{G}} \mathcal{A}$ with projections $\rho_x : \mathcal{C} \to \mathcal{A}$, and let $U:\mathcal{B} \to \mathcal{C}$ be define by $\rho_x U = U_x \pi_x$ for all $x \in \mathcal{G}$.  It is clear that $UP = UQ$, and $U$ is an algebraic, faithful, conservative, and amnestic functor, thus $P \approx Q$ is a Birkhoff equation defined on $\mathcal{A}$ which has the same solutions of $\{P_x \approx Q_x \}_{x \in \mathcal{G}}$. Therefore, $v$ is a general solution of $P \approx Q$, and $P \approx Q$ is a coreflexive pair since for any $x \in \mathcal{G}$ we have $\rho_xUP = \rho_x U Q =1_\mathcal{A}$. 
    \end{proof}

\begin{remark}
    Characterization of Birkhoff varieties \ref{Characterization of Birkhoff varieties} follows by Theorem \ref{main theorem} and Theorem \ref{regularmonobirkhoff}.
\end{remark}

In the following, we are going to use the theory of congruences on algebraic theories given in \cite[Chapter 10]{adámek_rosický_vitale_lawvere_2010}.

\begin{definition}\label{Lawvere equation}
    A \textbf{Lawvere equation} is an equation $P \approx Q$ in \textbf{AlgTh} such that there exists a morphism of theories $U$ surjective on objects such that $PU = QU$.
\end{definition}

Similarly to Birkhoff equations, Lawvere equations are closed under left composition by morphism of theories. Thus, the dual of Proposition \ref{Products and varieties} holds for Lawvere covarieties. 

\begin{theorem}\label{Lawvere covariety morphisms}
 Every cosystem of Lawvere equations has a general cosolution, Lawvere covarieties are precisely full morphisms of theories bijective on objects, and  each Lawvere covariety is a reflexive coequalizer of some reflexive Lawvere equation.  
\end{theorem}

\begin{proof}
    Let $\{P_i,Q_i:\mathcal{T}_i \to \mathcal{T}\}_{i \in I}$ be a cosystem of Lawvere equations. Note that if $P \approx Q$ is a Lawvere equation, then $P$ and $Q$ coincide on objects. Define $\sim$ as the congruence on $\mathcal{T}$ generated by the set of equations $E=\{P_i(g),Q_i(g) \mid i \in I \; \text{and} \; g \; \text{an arrow in} \; \mathcal{T}_i\}.$

    Let $M: \mathcal{T} \to \mathcal{T}/\sim$ be the canonical morphism of theories from $\mathcal{T}$ onto $\mathcal{T}/\sim$. Observe that $M$ is a full morphism of theories bijective on objects. By construction, it is clear that $M P_i = M Q_i$ for all $i \in I$. Now suppose $M':\mathcal{T} \to \mathcal{T}'$ is a morphisms of theories such that $M' P_i = M' Q_i$ for all $i \in I$. Thus, each equation in $E$ belongs to the congruence $\approx_{M'}$, so $\sim$ is contained in $\approx_M$ because $\sim$ is generated by $E$, then there exists a unique morphism of theories $N$ such that $NM = M'$. Therefore, $M$ is a general cosolution of the cosystem $\{P_i,Q_i:\mathcal{T}_i \to \mathcal{T}\}_{i \in I}$. 

    Next, suppose that $M:\mathcal{T} \to \mathcal{Q}$ is a full morphism of theories bijective on objects. Let $\pi_1,\pi_2:\mathcal{T} \times \mathcal{T} \to \mathcal{T}$ be the canonical projections, and let $\mathcal{R}$ be the subcategory of $\mathcal{T} \times \mathcal{T}$ whose objects are the pairs $(x,x)$, $x \in \mathcal{T}$, and $(f,g):(x,x) \to  (y,y)$ is a morphism in $\mathcal{R}$ if $M(f) = M(g)$. Let $J$ be the inclusion from $\mathcal{R}$ into $\mathcal{T} \times \mathcal{T}$, and let $J_i = \pi_i J$ for $i = 1,2$. Define $U:\mathcal{T} \to \mathcal{R}$ such that $U\left(x \xrightarrow{f} y \right) = (x,x) \xrightarrow{(f,f)} (y,y)$. It is easy to verify that $\mathcal{R}$ is an algebraic theory, $J_1 \approx J_2$ is a Lawvere coequation, $J_1,J_2$ is a kernel pair of $M$, $M$ is a coequalizer of $J_1,J_2$, and $J_1 U = J_2 U = 1_\mathcal{T}$. 
\end{proof}

\begin{remark}
    Characterization of Lawvere covarieties \ref{Characterization of Lawvere covarieties} follows by Theorem \ref{Lawvere covariety morphisms}.
\end{remark}

\begin{corollary}
    \textbf{AlgTh} is cowellpowered w.r.t. Lawvere covarieties. 
\end{corollary}

\begin{proof}
    Let $M:\mathcal{T} \to \mathcal{Q}$ be a Lawvere covariety, and let $M':\mathcal{T} \to \mathcal{T}/\approx_M$ be the canonical morphisms of theories. Then, there exists a unique morphism of theories $P$ such that $PM' = M$. Since $M$ is full and bijective on objects, we have $P$ is an isomorphism of theories, so $M$ and $M'$ are isomorphic. Therefore, the claim of this corollary is true since the class of all congruences on $\mathcal{T}$ is small. 
\end{proof}

To conclude this note we are going to study some relations between algebraic theories and algebraic categories, w.r.t. varieties and covarieties. 

\begin{lemma}\label{lemma Birkhoff Lawvere varieties}
    Let $\mathcal{T}$ be an algebraic theory and let $F$ be an algebraic functor with target $Alg \, \mathcal{T}$. Then  $F$ is isomorphic to  $\mathcal{V} \hookrightarrow Alg \, \mathcal{T}$ for some subvariety $\mathcal{V}$ of $Alg \, \mathcal{T}$ if only if $F$ is isomorphic to $Alg\, M :Alg \, \mathcal{Q} \to Alg \, \mathcal{T}$ for some Lawvere covariety $M:Alg \, \mathcal{T} \to Alg \, \mathcal{Q}$.
\end{lemma}

\begin{proof}
    The necessity follows by \cite[Corollary 10.15]{adámek_rosický_vitale_lawvere_2010} and Theorem \ref{Lawvere covariety morphisms}. Conversely, let $M':\mathcal{T} \to \mathcal{T}/\approx_{M}$, so $M$ and $M'$ are isomorphic, then $Alg \, M$ and $Alg \, M'$ are also isomorphic, and $Alg \, M'$ is isomorphic to $\mathcal{V} \hookrightarrow Alg \, \mathcal{T}$ for some subvariety $\mathcal{V}$ of $Alg \, \mathcal{T}$ by \cite[Corollary 10.15]{adámek_rosický_vitale_lawvere_2010}.
\end{proof}

\begin{lemma}\label{2do lemma Birkhoff Lawvere variety}
    Let $L:\mathcal{A} \to \mathcal{B}$ be an equivalence functor between algebraic categories. Then, for each subvariety $\mathcal{V}$ of $\mathcal{A}$ there exist a subvariety $\mathcal{W}$ of $\mathcal{B}$ and an equivalence functor $K:\mathcal{V} \to \mathcal{W}$ such that the following diagram is a pullback square 
    \[
    \begin{tikzcd}
\mathcal{V} \arrow[dd, "K"] \arrow[rr, "V", hook] &  & \mathcal{A} \arrow[dd, "L"] \\
                                                  &  &                             \\
\mathcal{W} \arrow[rr,"W", hook]                      &  & {\mathcal{B}}       
\end{tikzcd}
\]
\end{lemma}

\begin{proof}
    Denote by $V:\mathcal{V} \to \mathcal{A}$ to the inclusion of $\mathcal{V}$ into $\mathcal{A}$. Let $R:\mathcal{B} \to \mathcal{A}$ be an equivalence functor and let $(\eta,\varepsilon):L \vdash R$ be an adjoint equivalence. Define $\mathcal{W}$ as the full subcategory of $\mathcal{B}$ such that $x$ belongs to $\mathcal{W}$ if only if $Rx$ belongs to $\mathcal{V}$. Denote by $W:\mathcal{W} \hookrightarrow  \mathcal{B}$ to the inclusion of $\mathcal{W}$ into $\mathcal{B}$. Thus, we get a functor $J:\mathcal{W} \to \mathcal{V}$ such that $VJ = RW$. Observe that if $x \in \mathcal{V}$ then $Lx \in \mathcal{W}$. Indeed, since $\eta_x : x \to RLx$ is an isomorphism and $x \in \mathcal{V}$, then $RLx \in \mathcal{V}$, thus $Lx \in \mathcal{W}$. Therefore, we get a functor $K:\mathcal{V} \to \mathcal{W}$ such that $WK = LV$. Define natural transformations $\tau:1_\mathcal{V} \to JK$ and $\theta:KJ \to 1_\mathcal{W}$ such that $V \tau = \eta V$ and $W \theta = \varepsilon W$. It is easily seen that $(\tau,\theta):K \vdash J$ is an adjoint equivalence. 

    Now we are going to prove that $\mathcal{W}$ is a subvariety of $\mathcal{B}$. Let $x \in \mathcal{W}$ and $m:y \to x$ be a monomorphism in $Alg \, \mathcal{T}$. Thus, $Rm:Ry \to Rx$ is a monomorphism and $Rx \in \mathcal{V}$, then $Ry \in \mathcal{V}$ and $y \in \mathcal{W}$. Therefore, $\mathcal{W}$ is closed under subalgebras. Let $\{\nu_i : x_i \to y \}_{i \in D}$ be a limit cone in $\mathcal{B}$ such that $x_i \in \mathcal{W}$ for all $i \in D$. Thus, $\{Rv_i : Rx_i \to Ry \}_{i \in D}$ is a limit cone in $\mathcal{A}$ and $Rx_i \in \mathcal{V}_i$ for all $i \in D$, so $Ry \in \mathcal{V}$ and $y \in \mathcal{W}$. Therefore, $\mathcal{W}$ is closed under limits. Analogously we prove that $\mathcal{W}$ is closed under regular quotients and sifted colimits, so $\mathcal{W}$ is a subvariety of $\mathcal{B}$. It follows that $J$ and $K$ are algebraic functors. 

    Let us check that $\mathcal{W} \xleftarrow{J} \mathcal{V} \overset{V}{\hookrightarrow} \mathcal{A}$ is a pullback of $\mathcal{W} \overset{W}{\hookrightarrow} \mathcal{B} \xleftarrow{L} \mathcal{A}$. Suppose that $F:\mathcal{C} \to  \mathcal{A}$ and $G:\mathcal{C} \to \mathcal{W}$ are algebraic functors such that $WG=LF$. Let $x \in \mathcal{C}$ and $y = Fx$, so $Ly = WGx \in \mathcal{W}$, $RLy \in \mathcal{V}$, and $y \in \mathcal{V}$. Therefore, we get an algebraic functor $R:\mathcal{C} \to \mathcal{V}$ such that $F=VR$. Now, $WKR = LVR = L F = G W$, then $KR = G$ since $W$ is a monomorphism. It is clear that $R$ is unique with $F=VR$ and $G= KR$ due to $V$ is a monomorphism. 
\end{proof}

\begin{theorem}\label{Lawvere Birkhoff varieties}
    An algebraic functor $F:\mathcal{V} \to \mathcal{A}$ is a Birkhoff variety if only if there exist a Lawvere covariety $M:\mathcal{T} \to \mathcal{Q}$ and algebraic functors $\mathcal{B} \to Alg \, \mathcal{Q}$, $\mathcal{A} \to Alg \, \mathcal{T}$ such that following diagram is a pullback square
    \[
    \begin{tikzcd}
        \mathcal{B} \arrow[dd] \arrow[rr, "F"]        &  & \mathcal{A} \arrow[dd] \\
                                              &  &                        \\
        {Alg \, \mathcal{Q}} \arrow[rr, "{Alg \, M}"] &  & {Alg \, \mathcal{T}}  
    \end{tikzcd}
    \]
\end{theorem}

\begin{proof}
    Let us factorize $F$ as $F=VG$, where $\mathcal{V}$  is a subvariety of $\mathcal{A}$, $V:\mathcal{V} \hookrightarrow \mathcal{A}$ is the inclusion of $\mathcal{V}$ into $\mathcal{A}$, and $G:\mathcal{B} \to \mathcal{V}$ is an isomorphism. Since $\mathcal{A}$ is an algebraic category, there exists an algebraic theory $\mathcal{T}$ and an adjoint equivalence $(\eta,\varepsilon): L \vdash R$, where $R:Alg \,\mathcal{T} \to \mathcal{A}$ and $L:\mathcal{A} \to Alg \, \mathcal{T}$. 

    By Lemma \ref{2do lemma Birkhoff Lawvere variety} we found a subvariety $\mathcal{W}$ of $Alg \, \mathcal{T}$ and an algebraic functor $K:\mathcal{V} \to \mathcal{W}$ such that  $\mathcal{W} \xleftarrow{J} \mathcal{V} \overset{V}{\hookrightarrow} \mathcal{A}$ is a pullback of $\mathcal{W} \overset{W}{\hookrightarrow} Alg \, \mathcal{T} \xleftarrow{L} \mathcal{A}$. By Lemma \ref{lemma Birkhoff Lawvere varieties} we have a Lawvere covariety $M:\mathcal{T} \to \mathcal{Q}$ and an isomorphism $H: \mathcal{W} \to Alg \, \mathcal{Q}$ such that $W = Alg \, M H $. It is easy to check that  
      \[
    \begin{tikzcd}
\mathcal{B} \arrow[rr, "F"] \arrow[dd, "HKG"] &  & \mathcal{A} \arrow[dd, "L"] \\
                                              &  &                             \\
{Alg \, \mathcal{Q}} \arrow[rr, "{Alg \, M}"] &  & {Alg \, \mathcal{T}}       
\end{tikzcd}
    \]
    is a pullback square.

    Conversely, it follows by Lemma \ref{lemma Birkhoff Lawvere varieties} that $Alg \, M$ is a Birkhoff variety. Since Birkhoff varieties are closed under pullbacks (see Proposition \ref{closed class of equations}), we conclude that $F$ is also a Birkhoff variety. 
    \end{proof}

\begin{corollary}
\textbf{AlgCat} is wellpowered w.r.t. Birkhoff varieties.
\end{corollary}

\begin{proof}
For a given algebraic theory $\mathcal{T}$, the class of all subvarieties of $Alg \, \mathcal{T}$ is small because the the class of all morphisms in $\mathcal{T}$ is small. Then, the class of all Birkhoff varieties with target $Alg \, \mathcal{T}$ is small, up to isomorphisms, by Theorem \ref{main theorem}. Therefore, the claim of this corollary follows by Theorem \ref{Lawvere Birkhoff varieties}. 
\end{proof}

\section{Some open problems}

In the following list, we mention some open problems which are worth to investigate:

\begin{enumerate}
    \item In \textbf{Cat}, are the conservative embeddings precisely the regular monomorphisms?
    \item Is the class of Birkhoff equations the largest solution to the inverse main problem of varieties for the class of inclusions of varieties of algebraic categories? Or equivalently, is the class of Birkhoff equations closed under equivalences? Is the largest solution closed under action by algebraic functors?
    \item Could the results and proofs of the previous section to be generalize to the category of locally finitely presentable categories with right adjoint functors as morphisms, for an appropriate definition of varieties?

    \item With respect to algebraic geometry, it would be very appropriate as an application of this theory of varieties, to find a solution to the inverse main problem in the category of schemes over an algebraically closed field $k$, for the class of algebraic varieties over $k$, i.e., the class of reduced separated schemes of finite type over $k$. Examples \ref{forma differential varieties}, \ref{affine varieties} and \ref{projective varieties} are a motivation to this problem.  
\end{enumerate}

\bibliographystyle{amsalpha}
\bibliography{References}

\end{document}